\documentclass[a4paper]{amsart}

\usepackage{amssymb}
\usepackage{amstext}
\usepackage{amsmath}
\usepackage{mathptmx}
\usepackage{amscd}
\usepackage{amsthm}
\usepackage{amsfonts}
\usepackage{bm}
\usepackage{enumerate}
\usepackage{graphicx}
\usepackage{latexsym}
\usepackage{mathrsfs}
\usepackage{mathtools}
\usepackage{extarrows}
\usepackage{bbm}
\usepackage[colorlinks,pdfdisplaydoctitle,linkcolor=blue,citecolor=red]{hyperref}

\theoremstyle{definition}
\newtheorem{theorem}{Theorem}[section]

 \newtheorem{lemma}{Lemma}[section]
\newtheorem{rem}{Remark}[section]
\newtheorem{Corollary}{Corollary}[section]

 \newcommand{\ep}{\epsilon}
 \newcommand{\R}{\mathbb{R}}
  \newcommand{\cP}{\mathcal{P}}

  \newcommand{\ra}{\rightarrow}
 
 \newcommand{\dt}{\delta}

\begin{document}

\author[A. Liu]{Anning Liu }
\address[A. Liu]{Department of Mathematical Sciences, Tsinghua University, Beijing 100084, People's Republic
of China}
\email{lan15@mails.tsinghua.edu.cn}

\author[J. G. Liu]{Jian-Guo Liu}
\address[J. G. Liu]{Department of Mathematics and Department of Physics, Duke University, Durham NC 27708, USA}
\email{jliu@phy.duke.edu}

\author[Y. Lu]{Yulong Lu}
\address[Y. Lu]{Department of Mathematics, Duke University, Durham NC 27708, USA}
\email{yulonglu@math.duke.edu}

\title[Rate of convergence of empirical measure in $\infty$-Wasserstein distance]{On the rate of convergence of empirical measure in $\infty-$Wasserstein distance for unbounded density function}
\maketitle

\begin{abstract}
 We consider a sequence of identically independently distributed random samples from an absolutely continuous  probability measure in one dimension with unbounded density. We establish a new rate of convergence of the $\infty-$Wasserstein distance between the empirical measure of the samples and the true distribution, which extends the previous convergence result by Trilllos and Slep\v{c}ev to the case that the true distribution has an unbounded density.
\end{abstract}

\section{Introduction}

Consider a sequence of identically independently distributed (i.i.d.) random variables $\{X_i\}, i = 1,\cdots,n,$ sampled from a given probability measure $\nu \in \cP(\R^d)$ with probability density function $\rho$. Here $\cP(\R^d)$ denotes the space of all probability measures on $\R^d$. We define the empirical measure $\nu_n$ associated to  the samples $\{X_i\}$ by
$$\nu_n := \frac{1}{n}\sum^n_{i=1} \delta_{X_i}.$$
The well-known Glivenko--Cantelli theorem \cite{van1996weak} states that $\nu_n$ converges weakly to $ \nu$ as $n\ra \infty$. In recent years,
there has been growing interest in quantifying the rate of convergence of $\nu_n$ to $\nu$ with respect to Wasserstein distances. Recall that the $p$-Wasserstein distance between two probability measures $\mu, \nu \in \cP(\R^d)$ is defined as
$$
W_p(\mu, \nu) := \Big(\inf_{\gamma \in \Gamma(\mu, \nu)} \int_{\R^d \times \R^d} |x - y|^p \gamma(dx, dy)\Big)^{1/p}, \quad 1 \leq p < \infty
$$
and
$$W_{\infty}(\mu,\nu) := \inf_{\pi \in \Pi(\mu,\nu)}\textrm{esssup}_{\pi} | x  -y |,$$
where $ \Gamma(\mu, \nu)$ is the set of all probability measures on $\R^d \times \R^d$ with two marginals $\mu$ and $\nu$.

The purpose of this paper is to prove the rate of convergence of $\nu_n$ to $\nu$  w.r.t. $\infty$-Wasserstein distance $W_\infty$ when the density function $\rho$ of $\nu$ is unbounded. For simplicity, we will focus on the one dimensional case, but the arguments of the proof are expected to be generalized to high dimensions.

\bigskip

 \subsection{Motivation and Related Work}

Estimating the distance between the empirical measure
of a sequence of i.i.d. random variables and its true distribution  is a highly important problem in probability and statistics. For example, in statistics, it is usually impossible to access to the true distribution, e.g. the posterior distribution in a Bayesian procedure. So in order to extract useful information from the true distribution, a common approach is to generate i.i.d samples from the true distribution via various sampling algorithms (Markov chain Monte Carlo for instance), from which one can approximately
compute many statistical quantities of interest, such as the mean or variance by their empirical counterparts. Hence understanding the statistical error in estimating the statistics requires a quantification of the distance between the empirical measure and the true distribution.

The Wasserstein distance is a natural choice for measuring the closeness of two probability measures in the problem of consideration since it allows the probability measures to be singular to each other, which typically allows including Dirac masses or the empirical measures. This is prohibited if total variation distance or Hellinger distance\cite{gibbs2002choosing} are used. We are particularly interested in the $\infty$-Wasserstein distance for several reasons.  First, the $\infty$-Wasserstein distance $W_\infty(\mu, \nu)$ reduces to the so-called min-max matching distance \cite{Ajtai1984, barthe2013combinatorial,Leighton1989}  when both $\mu$ and $\nu$ are discrete measures with the same number of Diracs.
Such min-max matching distance plays an important role  in the analysis of shape matching problems in computer vision; see \cite{d2006using} and the references therein. Moreover, the $\infty$-Wasserstein distance is also useful in understanding the asymptotic performance of spectral clustering \cite{trillos2016variational, trillos2014rate}. In fact, in \cite{trillos2016variational}, the authors studied the consistency of spectral clustering algorithms in the large graph limit. By formulating the clustering procedure in a variational framework, they characterized  the convergence of eigenvalues, eigenvectors of a weighted graph Laplacian, and that of spectral clustering to their underlying continuum limits using $\Gamma$-convergence. One crucial ingredient needed in their proof is exactly a convergence rate estimate on the $\infty$-Wasserstein distance between the empirical measures and the true distribution which was established in \cite{trillos2014rate}. However, they made a strong assumption that the density function of  the true distribution is strictly bounded from above and below. We
aim to extend the result in \cite{trillos2014rate} to the case where the true distribution has an unbounded density in one dimensional space.

Let us briefly review some important previous works on the rate of convergence of $W_p(\nu_n, \nu)$ with $p \geq 1$.
For $p=1$, it was shown by Dudley  in \cite{dudley1969speed} that when $d \geq 2$,
$$C_2\cdot n^{-\frac{1}{d}} \leq \mathbb{E}(W_1(\nu,\nu_n)) \leq C_1\cdot n^{-\frac{1}{d}}. $$
Based on Sanov's theorem, Bolley, Guillin and Villani \cite{bolley2007quantitative} proved  a concentration estimate on $W_p(\nu_n, \nu)$ for $1\leq p \leq 2$  in any dimension
$$\mathbb{P}\big(W_p(\nu_n,\nu)\geq t\big) \leq C\cdot e^{-Knt^2}.$$
Boissard \cite{boissard2011simple} extended this result  to more general spaces rather than $\mathbb{R}^d$  when $p=1$ and applied it to  the occupation measure of a Markov chain.
In \cite{boissard2014mean}, Boissard and Gouic  gave the rate of convergence for $\mathbb{E}(W_p(\nu_n,\nu)^p)$ when $1 \leq p < \infty$.
Fournier and Guillin  \cite{fournier2015rate} presented a better result than \cite{bolley2007quantitative,boissard2011simple} for non-asymptotic moment estimates  and concentration estimates. They showed that if $\nu$ has finite $q$-th moment and $p < \frac{d}{2}$, then
 \begin{equation*}
   \mathbb{E}(W^p_p(\nu,\nu_n)) \leq C_{q,p}\cdot n^{-\frac{p}{d}},
 \end{equation*}
 and
 \begin{equation*}
   \mathbb{P}\big(W_p(\nu_n,\nu)\geq t\big) \leq C\cdot \exp(-Knt^{\frac{d}{p}}).
 \end{equation*}
 (We only list the case $p < \frac{d}{2}$ here. For other cases, one can refer to Theorem 1 and 2 in \cite{fournier2015rate}.)
 Weed and Bach gave a new definition of the upper Wasserstein dimension
 $d^{*}(\nu)$ for measure $\nu$. They proved that for $1 \leq p < \infty$ and $s < d^{*}(\nu)$,
 \begin{equation*}
   \mathbb{E}(W_p(\nu,\nu_n)) \leq C \cdot n^{-\frac{1}{s}}.
 \end{equation*}

As for $W_{\infty}(\nu,\nu_n)$, its rate of convergence is less studied  than that of $W_{p}(\nu,\nu_n)$ with $p < \infty$. As far as we know, most results on $W_{\infty}(\nu,\nu_n)$ are obtained when $\nu$ and $\nu_n$ are both discrete measures.   As mentioned above, the $\infty-$ Wasserstein distance between two discrete measures is closely linked to the min-max matching problem. Many results have been obtained for the latter when $\nu$ is a uniform distribution. Let $S = [0,1]^d$. Define a regularly spaced array of $n$ grid points on S (with $n = k^d$ for some $k\in \mathbb{N}$ ) by $Y_i$ and the i.i.d. random samples with uniform distribution on S by $X_i$.  Leighton and Shor \cite{Leighton1989}, and Yukich and Shor \cite{shor1991} showed that as $n \rightarrow \infty$,  it holds with high probability that
\begin{equation}\nonumber
 \min_{\pi}\max_i |X_{\pi_i} - Y_i| \sim
 \left\{
\begin{aligned}
   & O\left(\frac{(\log n)^{\frac{3}{4}}}{n^{\frac12}}\right), ~~~&d = 2,\\
   & O\left(\left(\frac{\log n}{n}\right)^{\frac1d}\right),~~~ &d \geq 3,
\end{aligned}
\right.
\end{equation}
where $\pi$ is a permutation of $\{1,2,\cdots,n\}$.
Trillos and Slep\v{c}ev \cite{trillos2014rate}  proved that the above estimate still holds when the underlying measure $\nu$ has a strictly positive and bounded density.

\bigskip

\subsection{Main Results}
 The purpose of this paper is to improve the results of \cite{trillos2014rate} in 1-D by  removing the boundedness constraint on $\rho(x)$. Our first result is a rate of convergence result in the case where the density function $\rho(x)$ is bounded from below, but not from above.

\bigskip

\begin{theorem}\label{w_1}
Let $D = (0,1) \subseteq \mathbb{R}$ and $\nu$ be a probability measure in D with a density function $\rho: D \rightarrow (0,\infty)$. Assume that there exists a constant $\lambda \in (0,1)$ such that
\begin{equation}\nonumber
\rho(x) \geq \lambda,~~\forall~~x\in D.
\end{equation}
Let $X_1,\cdots,X_n,\cdots$ be i.i.d. random variables sampled from $\nu$ and let $\nu_n$ be the corresponding empirical measure.
Then for any $ t > 0$,
\begin{equation*}
  \mathbb{P}\left(W_{\infty}(\nu,\nu_n)\geq \frac{t}{\lambda}\right) \leq 2\exp(-2nt^2).
\end{equation*}
In particular, for any constant $M > 1 $, except on a set with probability $\frac{1}{M}$,
\begin{equation}\label{result_w_1}
  W_{\infty}(\nu,\nu_n)\leq \frac{1}{\lambda} \cdot \left(\frac{\log(2M)}{2n}\right)^{\frac12}.
\end{equation}
\end{theorem}

\bigskip

\begin{rem}
Note that the right hand side of (\ref{result_w_1}) will blow up if $\lambda \rightarrow 0$. That's why we assume that $\rho(x)$ has a uniform positive lower bound in Theorem \ref{w_1}.
Moreover, the exponent one half is sharp owing to the central limit theorem. \end{rem}

\bigskip

We proceed to discussing the case when the density function is not strictly bounded away from zero. We first comment that if the density function of $\nu$ is zero in a connected region, then by definition the $\infty$-Wasserstein distance between $\nu_n$ and $\nu$ can not go to zero as $n$ goes to infinity. In fact, consider the probability measure $\nu_0$ with the density function
\begin{equation}\nonumber
\rho_0(x) = \left\{
\begin{aligned}
  \frac32,&~~~x\in \left(0,\frac13\right)\bigcup \left(\frac23,1\right),\\
  0,&~~~x\in\left[\frac13,\frac23\right].
\end{aligned}
\right.
\end{equation}
Let $\nu_{n,0} $ be the empirical measure of $\nu_0$.
Since $\nu_{n,0}$ depends on a sequence  of random variables, there is no guarantee that $\nu_{n,0}((0,\frac13)) = \nu_0((0,\frac13)).$
Assume that $\nu_{n,0}((0,\frac13)) = \nu_0((0,\frac13)) + \delta_n $, where $1 \gg \delta_n > 0$ is a small parameter. Since $W_{\infty}(\nu_{n,0},\nu_n)$ is also the maximal distance that an optimal transportation map from $\nu_{n,0}$ to $\nu_0$ moves the mass by (which will be mentioned later in Lemma \ref{Anoth_def}), it follows that $$W_{\infty}(\nu_{n,0},\nu_0) \geq \textrm{diam}\left(\left(\frac13,\frac23\right)\right) = \frac13 > 0.$$

Therefore, in Theorem \ref{w_2} below, we assume that $\rho(x)$ only has a finite number of zero points.

\bigskip

\begin{theorem}\label{w_2}
Let $D = (0,1) \subseteq \mathbb{R}$ and $\nu$ be a probability measure in D with a density function $\rho: D \rightarrow (0,\infty)$.
Assume that there exists a constant $\Lambda > 0$ such that for all $x \in D$,
\begin{equation*}
\rho(x) \leq \Lambda.
\end{equation*}
Suppose additionally that there are only N points $x_1,\cdots,x_N$ satisfying  $\rho(x_i) = 0$. For each $x_i$, further assume that
   \begin{equation}\label{control_rho}
     \underline{C_i}|x_i - x|^{k_i} \leq \rho(x) \leq \overline{C_i}|x_i - x|^{k_i},~~\textrm{for}~~\forall~~x\in B_i,
   \end{equation}
where $B_i = (x_i - \Delta_i, x_i + \Delta_i)$ is a small neighborhood of $x_i$ and $ \Delta_i, k_i, \overline{C_i}, \underline{C_i}$ are positive numbers, $k_i \in \mathbb{Z}$.
Let $X_1,\cdots,X_n,\cdots$ be i.i.d. random variables sampled from $\nu$ and let $\nu_n$ be the corresponding empirical measure.
Then there exists a positive constant $C = C ( k_i, \underline{C_i} )$ such that except on a set with probability $O\left(\frac{1}{\log n}\right),$
$$W_\infty(\nu,\nu_n) \leq C\cdot\max_i \Big(\frac{\log n }{n}\Big)^{\frac{1}{2(k_i+1)}}.$$
\end{theorem}

\bigskip

We would like to sketch the proof of the theorems above. To prove Theorem \ref{w_1}, we use the fact that in one dimension, $W_{\infty}$ distance between two measures can be written as the $L^{\infty}$ norm of the difference of their quantile functions. Moreover, thanks to the $\frac{1}{\lambda}$- Lipschitz continuity of the quantile function of $\nu$, which follows by the assumption that $\rho \geq \lambda$,  the $L^{\infty}$ norm of the difference of the quantile functions can be bounded from above by the
difference between the cumulative distribution function of the true distribution $\nu$ and that of the empirical distribution $\nu_n$. Finally, the latter can be bounded by using the Dvoretzky-Kiefer-Wolfowitz inequality \cite{dvoretzky1956asymptotic}.

For the proof of Theorem \ref{w_2},  we first divide the domain $D$ into a family of sub-domains according to the value of $\rho(x)$. Then, we use the following scaling equality in each sub-domain
\begin{align*}
  W_{\infty}(\nu,\nu_n) = W_{\infty}(\theta\nu,\theta\nu_n),  
\end{align*}
with an appropriate scaling parameter $\theta$ such that after rescaling, the Lebesgue density of the rescaled measure $\theta \nu$ is bounded from above and below. With the density being bounded, we can estimate the $\infty$-Wasserstein distance by using the same method in \cite{trillos2014rate}. However, the mass of $\nu$ and $\nu_n$ may not be equal in each sub-domain. To resolve this issue, we introduce a new measure $\widetilde{\nu}$ such that $\widetilde{\nu}$ has the same mass as $\nu_n$ in each sub-domain. Since the distance between $\widetilde{\nu}$ and $\nu_n$ can be bounded by an argument similar to Theorem 1.1 in \cite{trillos2014rate}, it suffices to estimate the distance between $\nu$ and $\widetilde{\nu}$.

The following corollary is a direct consequence of  Theorem \ref{w_1} and Theorem \ref{w_2}.

\bigskip

\begin{Corollary}\label{w_3}
Let $D = (0,1) \subseteq \mathbb{R}$ and $\nu$ be a probability measure in D with density $\rho: D \rightarrow (0,\infty)$. Assume that there are only N points $x_1,\cdots,x_N$ satisfying $\rho(x_i) = 0.$ For each $x_i$, $\rho(x)$ satisfies (\ref{control_rho}).
Let $X_1,\cdots,X_n,\cdots$ be i.i.d. random variables sampled from $\nu$.
Then there exists a positive constant $C = C ( \dt, M, k_i, \underline{C_i} )$ such that except on a set with probability $O\left(\frac{1}{\log n}\right),$
\begin{equation}\label{eq:rate}
W_\infty(\nu,\nu_n) \leq C \cdot \max_i \Big(\frac{\log n }{n}\Big)^{\frac{1}{2(k_i+1)}}.
\end{equation}
\end{Corollary}

\bigskip

\subsection{Discussion}
As we mentioned earlier, quantifying the rate of convergence of $\nu_n$ to $\nu$ with respect to $\infty-$Wasserstein distance is very useful for understanding the consistency of spectral clustering\cite{trillos2016variational}. Our new convergence rate estimates will reshape the convergence of spectral clustering in the case where the density of true distribution is unbounded, as we discuss in what follows.

Let $V = \{x_1,\cdots,x_n\}$ be a set of data points in $\R^d$ sampled from a probability measure $\nu$. For each pair of points $x_i$ and $x_j$, we construct a weight $W^{\ep_n}_{i,j}$ between them to characterize their similarities. In general, the weight has the form of
\begin{equation*}
W^{\ep_n}_{i,j} = \eta_{\ep_n}(x_i - x_j),
\end{equation*}
where $\eta_{\ep_n}(z) = \frac{1}{\ep_n^d} \eta(\frac{z}{\ep_n})$ and $\eta$ is an appropriate kernel function(for example, Gaussian kernel). The weight matrix $W^{\ep_n}\in \R^{n\times n}$ is then defined by $W^{\ep_n}_{i,j}$.
Let $D^{\ep_n} \in \R^{n\times n}$ be a diagonal matrix with $D^{\ep_n}_{ii} = \sum_{j}W^{\ep_n}_{i,j}$. Then the discrete Dirichlet energy  and the relevant continuum Dirichlet energy are defined by
\begin{equation*}
  G_{n,\ep_n}(u) = \frac{1}{\ep^2_n n^2} \sum_{i,j}W^{\ep_n}_{i,j}\left(u\left(x_i\right) - u\left(x_j\right)\right)^2
\end{equation*}
and
\begin{equation*}
G(u) = \int_D |\nabla u|^2 \rho^2(x) dx,
\end{equation*}
where $\rho(x)$ is the density function of the underlying measure $\nu$.
 The unnormalized graph Laplacian $L_{n,\epsilon_n}$ is defined by
\begin{equation*}
L_{n,\ep_n} = D^{\ep_n} - W^{\ep_n}.
\end{equation*}

The aim of spectral clustering is to partition the data points $x_1,\cdots,x_n$ into $k$ meaningful groups. To do this, the spectrum of unnormalized graph Laplacian $L_{n,\epsilon_n}$ is used to embed the data points into a low dimensional space. Then we can apply some clustering algorithms like k-means to these points.
For more details about spectral clustering, one can see \cite{von2007tutorial}.

In \cite{trillos2016variational}, the authors proved that when the density function $\rho(x)$ of $\nu$ is bounded from above and below, the spectrum of unnormalized graph Laplacian $L_{n,\ep_n}$ converges to the spectrum of the corresponding continuum operator $L$, which implies the consistency of spectral clustering. They also gave a lower bound of the convergence rate at which the connectivity radius $\ep_n \rightarrow 0$ as $n \rightarrow \infty$. With our theorems, the results in \cite{trillos2016variational} can be generalized to the case when $\rho(x)$ is unbounded. In particular, the kernel width $\ep_n$ should be chosen to be slightly bigger than the right side of \eqref{eq:rate}, which is different from \cite{trillos2016variational}.

The proof will not be included in this paper since it is similar to the proof in \cite{trillos2016variational}. We sketch the outline of the proof as follows: First, we prove the $\Gamma-$ convergence of Dirichlet energy $G_{n,\ep_n}$ to $G$. Our theorems are used in this step to establish the probabilistic estimates and the constraint on $\ep_n$. Next, by min-max theorem, we know that the eigenvalues of $L_{n,\ep_n}(\textrm{or } L)$ can be written as the minimizers of $G_{n,\ep_n}(\textrm{or } G)$. Therefore, the convergence of spectrum is equivalent to the convergence of the minimizers of $G_{n,\ep_n}$ , which can be proved by the $\Gamma-$ convergence and compactness properties of $G_{n,\ep_n}$. Finally, with the convergence of spectrum, we can prove the consistency of spectral clustering.


The paper is organized as follows: In section 2, we introduce some preliminaries and notations. In section 3.1 and section 3.2,  we prove Theorem \ref{w_1} and Theorem \ref{w_2} respectively.  Finally, the proof of Corollary \ref{w_3} is presented in section 3.3.

\section{Preliminaries and notations}

\subsection{Notations}
Let $D = (0,1) \subset \R$ and $\mathcal{P}(D)$ be the set of all probability measures on $D$. Given a probability measure $\mu \in \mathcal{P}(D)$ and a Borel-measurable map $T$, we define the pushforward $\nu$ of measure $\mu$ under the map $T$ by setting
$$\nu (A) = T_{\sharp}\mu (A) = \mu(T^{-1}(A))$$
for any measurable set $A \subset D$. We call $T$ the transportation map between $\mu$ and $\nu$.

The $\infty-$Wasserstein distance $W_\infty(\mu,\nu)$ is defined by
$$W_{\infty}(\mu,\nu) = \inf_{\pi \in \Pi(\mu,\nu)}\textrm{esssup}_{\pi} | x  -y |,$$
  where $\Pi(\mu,\nu)$ is the set of all couplings between $\mu$ and $\nu$, i.e.
  $$\begin{aligned}
  \Pi(\mu,\nu) & = \Big\{\pi \in \mathcal{P}((0,1)^2) | \pi(A \times (0,1)) = \mu(A),\\
  & \qquad \pi((0,1) \times B) = \nu(B),  \text{ for all Borel sets } A,B \subset  (0,1)\Big\}.
  \end{aligned}
  $$

 \bigskip

  \begin{rem}
      Note that the definition of $W_{\infty}(\mu,\nu)$ can be generalized  to the case where $\mu$ and $\nu$ have the same mass on $D$. Therefore, in the sequential, we still write $W_{\infty}(\mu,\nu)$ when $\mu(D) = \nu(D)$ even though $\mu$ and $\nu$ are not necessarily probability measures.
  \end{rem}

\bigskip

  It was proved in \cite{champion2008wasserstein} that  if $\mu$ is absolutely continuous with respect to the Lebesgue measure, then for any optimal transport plan $\pi$ of $W_{\infty}(\mu,\nu)$, there exists a transportation map $ T : D \ra D $ such that $ T_{\sharp}\mu =\nu$ and $ \pi = (I \times T )_{\sharp} \mu $ .
  In particular, the optimal transportation plan of $W_{\infty}(\nu,\nu_n)$, with $\nu_n$ being the empirical measure of  the absolutely continuous probability measure $\nu$ is unique.

\subsection{Useful lemmas}
The following lemma collects some properties on $W_{\infty}$ to be used in subsequent sections. The proof is trivial and  thus is omitted.

\bigskip

\begin{lemma}\label{infiniteW}
Given measures $\mu_1,\mu_2,\mu_3 $ defined on $D$ with $\mu_1(D) = \mu_2(D) = \mu_3(D)$,
  then the followings hold:
\begin{enumerate}
\item
Triangle inequality: $W_{\infty}(\mu_1,\mu_3) \leq W_{\infty}(\mu_1,\mu_2) + W_{\infty}(\mu_2,\mu_3)$.

\item
Scaling equality: $W_{\infty}(\mu_1,\mu_2) = W_{\infty}(\alpha \mu_1, \alpha \mu_2)$, for $\forall \alpha >0$.

\item
$W_{\infty}(\mu_1,\mu_2) \leq \textrm{diam}(D) $.

\item
If $D = \bigsqcup_j D_j$  then
    $$W_{\infty}(\mu_1,\mu_2) \leq \max_j W_{\infty}(\mu_1|_{D_j},\mu_2|_{D_j}).$$
    \end{enumerate}
\end{lemma}

\bigskip

The following two lemmas gives two different characterizations of $W_{\infty}(\mu, \nu)$.

\bigskip

\begin{lemma}[\cite{champion2008wasserstein}]\label{Anoth_def}
  Let $\mu, \nu$ be two Borel measures   with $\mu$ absolutely continuous with respect to the Lebesgue measure and $\mu(D) = \nu(D)$. Then there exists an optimal transportation map $T: D \ra D$ such that $T_{\sharp}\mu =\nu$ and 
  $$W_{\infty}(\mu,\nu) = \|I - T\|_{L^{\infty}(D)}.$$
  Furthermore, if $\nu = \Sigma_{i=1}^{k}a_i \delta_{y_i}$ with $y_i \in D$ and positive numbers $a_i, i = 1,\cdots, k$, then there exists a unique transportation map $T^{\star}: D \ra D$ such that
  $$W_{\infty}(\mu,\nu) = \|I - T^{\star}\|_{L^{\infty}(D)}.$$
\end{lemma}

\bigskip

\begin{lemma}[{{\cite[Remark 2.19]{villani2003topics}}}]\label{lemma_quan}
  Let $\mu$, $\nu$ be two probability measures on $\R$. Denote the cumulative distribution functions of $\mu$ and $\nu$ by $F(x)$ and $G(x)$ respectively. Then we have the following equality that
  \begin{equation*}
    W_{\infty}(\mu,\nu) = \left\|F^{-1} - G^{-1}\right\|_{L^{\infty}}.
  \end{equation*}
\end{lemma}

\bigskip
\begin{lemma}[{{\cite[Lemma 2.2]{trillos2014rate}}}]\label{inf_den}
Let  $\nu_1$ and $\nu_2$ be two probability measures defined on $D$ with density functions $\rho_1(x)$ and $\rho_2(x)$ respectively. Assume that there exists a positive constant $\lambda > 0$ such that
$$\rho_i(x) \geq \lambda >0 ,~i=1,2.$$
Then there exists $C> 0 $ such that
$$W_{\infty}(\nu_1,\nu_2) \leq \frac{C}{\lambda}\cdot \textrm{diam}(D)\|\rho_1(x)-\rho_2(x)\|_{L^{\infty}(D)}.$$
\end{lemma}

\bigskip

The following three probability inequalities on binomial random variables and the Dvoretzky-Kiefer-Wolfowitz inequality will be used in the proofs of main results.

\bigskip

\begin{lemma}
Let $S_n \sim \textrm{Bin}(n,p)$ be the independent binomial random variables. For $t > 0$, Chebychev's inequality\cite{chebichef1867} states that
\begin{equation*}
  \mathbb{P}\left(\frac{|S_n - n\cdot p|}{\sqrt{np(1-p)}}  \geq t\right) \leq \frac{1}{t^2}.
\end{equation*}
The Chernoff's inequality \cite{Chernoff_1952} states that
\begin{equation*}
  \mathbb{P}\Big(\big|\frac{S_n}{n} - p \big| \geq t\Big) \leq 2 \exp(-2nt^2).
\end{equation*}
Bernstein's inequality \cite{Bernstein1927} states that
\begin{equation*}
\mathbb{P}\Big(\big|\frac{S_n}{n} - p \big| \geq t\Big) \leq 2 \exp\Big(-\frac{\frac12 n^2t^2}{np(1-p)+\frac13 nt}\Big).
\end{equation*}

\end{lemma}

\bigskip

\begin{lemma}[ Dvoretzky-Kiefer-Wolfowitz inequality \cite{dvoretzky1956asymptotic}]
  Let $\{X_i\}_{i=1}^n$ be the i.i.d. random variables sampled from a probability measure $\nu$. Let $F(x)$ be the cumulative distribution function of $\nu$ and $F_n(x)$ be the cumulative distribution function of $\nu_n$. Then for $\forall t > 0$,
  \begin{equation*}
    \mathbb{P}\left(\sup_x|F_n(x) - F(x)| \geq t\right) \leq 2\exp(-2nt^2).
  \end{equation*}
\end{lemma}

\bigskip

\section{Convergence of empirical measure}

\subsection{Proof of Theorem \ref{w_1} }

\begin{proof}

Denote the cumulative distribution function of $\nu_n$ by $F_n(x)$ and that of $\nu$ by $F(x)$. Thanks to the Dvoretzky-Kiefer-Wolfowitz inequality \cite{dvoretzky1956asymptotic},
\begin{equation*}
  \mathbb{P}\left(\sup_x|F_n(x) - F(x)| \geq t\right) \leq 2\exp(-2nt^2).
\end{equation*}
From this, we claim that
 \begin{equation}\label{claim_1}
   \mathbb{P}\left(\sup_y\left|F_n^{-1}(y) - F^{-1}(y)\right| \geq \frac{t}{\lambda}\right) \leq \mathbb{P}\left(\sup_x\left|F_n(x) - F(x)\right| \geq t\right)
 \end{equation}
which implies that
\begin{equation*}
  \mathbb{P}\left(\sup_y\left|F_n^{-1}(y) - F^{-1}(y)\right| \geq \frac{t}{\lambda}\right) \leq 2\exp(-2nt^2).
\end{equation*}

To prove \eqref{claim_1}, it suffices to show that $\sup_x \left|F_n(x) - F(x)\right| \leq t$ implies $\sup_y \left|F_n^{-1}(y) - F^{-1}(y)\right| \leq \frac{t}{\lambda}$. To this end,
fix $ y\in [0,1]$. Let $x_1 = F_n^{-1}(y) $ and $x_2 = F^{-1}(y) $. Then from the fact that the density function $\rho(x)$ has a lower bound $\lambda$ we know that
\begin{equation*}
  \frac{\left|F(x_1) - F(x_2)\right|}{|x_1 - x_2|} \geq \lambda.
\end{equation*}
It follows that
\begin{equation*}
  \lambda  |x_1 - x_2| \leq \left|F(x_1) - F(x_2)\right| = \left|F(x_1) - F_n(x_1)\right| \leq t,
\end{equation*}
where the last inequality is obtained from $\sup_x \left|F_n(x) - F(x)\right| \leq t$. Therefore, for any $ y$,
\begin{equation*}
  \left|F_n^{-1}(y) - F^{-1}(y)\right| \leq \frac{t}{\lambda}.
\end{equation*}
which completes the proof of \eqref{claim_1}. It follows from \eqref{claim_1} and Lemma \ref{lemma_quan} that
\begin{equation*}
     \mathbb{P}\left(W_{\infty}(\nu,\nu_n) \geq \frac{t}{\lambda}\right) \leq 2\exp(-2nt^2).
\end{equation*}

By taking $t =  \left(\frac{\log(2M)}{2n}\right)^{\frac12} $ we get that except on a set with probability $\frac{1}{M}$,
\begin{equation*}
  W_{\infty}(\nu,\nu_n)\leq \frac{1}{\lambda} \cdot \left(\frac{\log(2M)}{2n}\right)^{\frac12}.
\end{equation*}

\end{proof}

\bigskip

\subsection{Proof of Theorem \ref{w_2}}

\begin{lemma}\label{1lamma}
  If $a>b>0$, then $a^k -b^k \geq (a-b)^k$.
\end{lemma}

\begin{proof}
  By induction, we only need to prove that $a^k -b^k \geq (a-b)^k$ implies $a^{k+1} -b^{k+1} \geq (a-b)^{k+1}$.
  From $a>b>0$ we know $2b^{k+1} \leq ab^k +ba^k$. Therefore,
  \begin{equation}
      (a-b)^{k+1} = (a-b)(a-b)^k \leq (a-b)(a^k -b^k) = a^{k+1} +b^{k+1} - ab^k- ba^k \leq a^{k+1} - b^{k+1}.
  \end{equation}
\end{proof}

In Theorem \ref{w_2}, we give the rate of convergence of $W_{\infty}(\nu_n,\nu)$ when the density function $\rho(x)$ is not strictly bounded away from zero. The proof is a refinement of the proof of  \cite[Theorem 1.1]{trillos2014rate}, which deals with the case where $\rho(x)$ is bounded. We sketch the rough idea of our proof in the followings before we give the details.

To prove the theorem, we would like to use Lemma  \ref{infiniteW}-(4) to reduce the estimate of $W_{\infty}(\nu,\nu_n)$ to that of $W_{\infty}(\nu |_{B_i},\nu_n|_{B_i})$, where $B_i$ is a small neighborhood of the zero point $x_i$ . For doing so, we
 need to modify the measure $\nu$ locally (denote the new measure to be $\tilde{\nu}$ after modification) so that $\tilde{\nu}$ has the same mass as $\nu_n$ on $B_i$.
Then, we divide $B_i$ into a family of sub-domains $\{A_j\}_{j\in \mathbb{N}}$ according to the value of $\rho(x)$ so that $\rho$ is bounded from above and below on $A_j$. Thus we can adapt similar arguments from \cite{trillos2014rate} to obtain bounds on
 $W_{\infty}(\widetilde{\nu} |_{A_j},\nu_n|_{A_j})$.  However, $\nu_n$ may not have the same mass as $\widetilde{\nu}$ on each $A_j$. So, in order to remove this mass discrepancy, we introduce another new measure $\overline{\nu}$ such that $\overline{\nu}(A_j) = \nu_n(A_j)$. At last, thanks to Lemma \ref{infiniteW}, we can establish an upper bound on $W_{\infty}(\overline{\nu}|_{B_i},\nu_n|_{B_i})$ with the estimates of  $W_{\infty}(\overline{\nu}|_{A_j},\nu_n|_{A_j})$.

\begin{proof}
   Let $B_{N+1} = (0,1)\backslash \cup_1^N B_i$. Then $\{B_i\}_{i = 1}^{N+1}$ is a partition of $D$. Let $\epsilon_i = \frac{\nu_n(B_i)}{\nu(B_i)}-1$ for $i = 1,\cdots,N+1$ and $\widetilde{\nu}$ be a probability measure defined on $D$
   \begin{equation}\label{nu:widetilde}
     d\widetilde{\nu} = \left(\sum_{i=1}^{N+1}(1+\epsilon_i) \rho(x) \mathbbm{1}_{B_i}\right)dx.
   \end{equation}
    Then it's clear that
    \begin{equation*}
      \widetilde{\nu}(B_i)=(1+\epsilon_i)\nu(B_i) = \nu_n(B_i).
    \end{equation*}
    Combining this with Lemma $\ref{infiniteW}$, we obtain that
    $$W_{\infty}(\nu,\nu_n) \leq W_\infty(\nu,\widetilde{\nu}) + W_\infty(\widetilde{\nu},\nu_n) \leq W_\infty(\nu,\widetilde{\nu}) + \max_{i = 1,\cdots,N+1} W_\infty(\widetilde{\nu}|_{B_i},\nu_n|_{B_i}).$$

   Choose $\beta>2$. To estimate $W_\infty(\widetilde{\nu}|_{B_i},\nu_n|_{B_i}) (i = 1,\cdots,N)$, we divide $B_i$ into a family of sub-domains $\{A_j\}_{j\in \mathbb{N}}$ and use scaling property to bound $W_{\infty}$ distance on each sub-domain $A_j$.

   Define $\{A_j\}_{j\in \mathbb{N}}$ by $A_0 = \{x: 1 < \rho(x) \leq \Lambda \} \bigcap B_i$ , $A_{j} =\left\{x: \frac{1}{(j+1)^{\beta}} < \rho(x) \leq \frac{1}{j^{\beta}}\right\} \bigcap B_i$ (If $A_j$ is empty, just neglect it).
   Then,$$B_i = \bigsqcup_j A_j.$$

   Let $\delta_j = \frac{\nu_n(A_j)}{\nu(A_j)}-1$ and define a measure $\overline{\nu}$ on $B_i$ by
   $$d \overline{\nu} = \sum_j \mathbbm{1}_{A_j}(1+\delta_j)\rho(x)dx.$$
    Then it's easy to see that
    \begin{equation*}
      \overline{\nu}(A_j) = (1+\delta_j)\nu(A_j) = \nu_n(A_j).
    \end{equation*}
   Again, with this and Lemma \ref{infiniteW}, we can bound $W_\infty(\widetilde{\nu}|_{B_i},\nu_n|_{B_i})(i = 1,\cdots,N)$ as follows
\begin{equation*}
  \begin{aligned}
  W_\infty(\widetilde{\nu}|_{B_i},\nu_n|_{B_i}) &\leq W_\infty(\widetilde{\nu}|_{B_i},\overline{\nu}|_{B_i}) +W_\infty(\overline{\nu}|_{B_i},\nu_n|_{B_i}) \\
&\leq W_\infty(\widetilde{\nu}|_{B_i},\overline{\nu}|_{B_i}) +\sup_j W_\infty(\overline{\nu}|_{A_j},\nu_n|_{A_j}).
  \end{aligned}
\end{equation*}
Therefore, to estimate $W_{\infty}(\nu,\nu_n)$, it suffices to estimate $W_\infty\left(\nu,\widetilde{\nu}\right)$, $W_\infty\left(\widetilde{\nu}|_{B_i},\overline{\nu}|_{B_i}\right)$, $W_\infty\left(\overline{\nu}|_{A_j},\nu_n|_{A_j}\right)$, and $ W_\infty\left(\widetilde{\nu}|_{B_{N+1}},\nu_n|_{B_{N+1}}\right)$ respectively.

\bigskip

\textbf{\emph{Step 1:}}\
We first estimate $W_\infty\left(\widetilde{\nu}|_{B_{N+1}},\nu_n|_{B_{N+1}}\right).$ It's easy to deduce, via Lemma \ref{infiniteW}, that
\begin{equation}\nonumber
\begin{aligned}
    W_{\infty}\left(\widetilde{\nu}|_{B_{N+1}},\nu_n|_{B_{N+1}}\right) &= W_{\infty}\left(\frac{1}{\widetilde{\nu}(B_{N+1})}\widetilde{\nu}|_{B_{N+1}}, \frac{1}{\widetilde{\nu}(B_{N+1})}\nu_n|_{B_{N+1}}\right)\\
  & = W_{\infty}\left(\frac{1}{\nu(B_{N+1})} \nu|_{B_{N+1}},\frac{1}{\sum\delta_{X_i}(B_{N+1})}\sum_{X_i \in B_{N+1}}\delta_{X_i}|_{B_{N+1}}\right).
\end{aligned}
\end{equation}
To ease the notations, we write $\nu_{{N+1}} = \frac{1}{\nu(B_{N+1})} \nu|_{B_{N+1}}$ and $\nu_{n,N+1} =  \frac{1}{\sum\delta_{X_i}(B_{N+1})}\sum_{X_i \in B_{N+1}}\delta_{X_i}|_{B_{N+1}}.$

Clearly, $\nu_{N+1}$ is the restriction of $\nu$ to $B_{N+1}$ and $\nu_{n,N+1}$ is the empirical measure of $\nu_{N+1}$. Furthermore, we note that $\rho(x)$ is bounded from below in $B_{N+1}$ due to the fact that $B_i(i = 1,\cdots,N)$ is a small neighborhood of zero point $x_i$ and $B_{N+1} = D\setminus\bigcup_1^N B_i$.
 Therefore, we can use Theorem \ref{w_1} to give an estimate on $W_\infty\left(\widetilde{\nu}|_{B_{N+1}},\nu_n|_{B_{N+1}}\right).$(We remark that Theorem \ref{w_1}  holds true for any domain $(a,b) \subset \R$ by replacing $D = (0,1)$ with $D = (a,b)$ in the proof.)

   Let $\lambda_{N+1} := \min_{x\in B_{N+1}} \rho(x)$. Then we have $0 < \frac{\lambda_N}{\nu(B_{N+1})} \leq\frac{1}{\nu(B_{N+1})} \rho(x)|_{B_{N+1}}$. It follows from Theorem \ref{w_1} that there exists a constant $C = \frac{\nu(B_{N+1})}{\lambda_N}$ such that
   \begin{equation*}
     W_{\infty}\left(\widetilde{\nu}|_{B_{N+1}},\nu_n|_{B_{N+1}}\right) \leq C \left(\frac{\log n}{n}\right)^{\frac12}.
   \end{equation*}

\bigskip

\textbf{\emph{Step 2:}} \
We then estimate $W_\infty\left(\overline{\nu}|_{A_j},\nu_n|_{A_j}\right)$.
To achieve this, set   $$J_0  = \left\lfloor(\frac{n}{\log n})^{\frac{k_i}{2\beta(k_i+1)}}\right\rfloor - 1$$
and consider the following two cases: 1) $j<J_0$ and 2) $j\geq J_0$.

We claim that,
when $j \geq J_0$, $\textrm{diam}(A_j) \leq C \cdot \left(\frac{\log n}{n}\right)^{\frac{1}{2(k_i+1)}}$.
To show the claim, we first recall the definition that $A_{j} =\left\{x: \frac{1}{(j+1)^{\beta}} < \rho(x) \leq \frac{1}{j^{\beta}}\right\} \bigcap B_i~$
and the assumption that
$~\underline{C_i}|x_i - x|^{k_i} \leq \rho(x) \leq \overline{C_i}|x_i - x|^{k_i}$ in $B_i$.  To simplify the notations, we denote  $\underline{C}_i|x-x_i|^{k_i}$ by $\rho_1(x)$ and $\overline{C}_i|x-x_i|^{k_i}$ by $\rho_2(x) $. Let $x_R$ and $x_L$ be positive constants satisfying $\rho_1(x_R) = \frac{1}{j^{\beta}}$, $\rho_2(x_L) = \frac{1}{(j+1)^{\beta}}.$

From $\rho_1(x) \leq \rho(x) \leq \rho_2(x)$ we know that $\textrm{diam}(A_j) \leq x_R-x_L \leq x_R.$ Moreover, when $n$ is large enough,
\begin{equation*}
 \begin{aligned}
x_R &= \frac{1}{{(\underline{C}_i)}^{\frac{1}{k_i}}} \cdot j^{-\frac{\beta}{k_i}} \leq C \cdot J_0^{-\frac{\beta}{k_i}} = C \cdot \left(\left\lfloor\left(\frac{n}{\log n}\right) ^{\frac{k_i}{2\beta(k_i+1)}} \right\rfloor - 1\right) ^{-\frac{\beta}{k_i}} \\
& \leq C \cdot \left(\frac12 \cdot \left(\frac{n}{\log n}\right)^{\frac{k_i}{2\beta(k_i+1)}} \right)^ {-\frac{\beta}{k_i}}\\
  &\leq C \cdot \left(\frac{\log n}{n}\right)^{\frac{1}{2(k_i+1)}},
\end{aligned}
\end{equation*}
where $C = C(k_i,\underline{C}_i,\beta) $ .

Therefore, when $j \geq J_0$, $\textrm{diam}(A_j) \leq C \cdot \left(\frac{\log n}{n}\right)^{\frac{1}{2(k_i+1)}} $. By Lemma \ref{infiniteW}, when $j \geq J_0$,
$$W_\infty(\overline{\nu}|_{A_j},\nu_n|_{A_j})  \leq \textrm{diam}(A_j)  \leq C \cdot \left(\frac{\log n}{n}\right)^{\frac{1}{2(k_i+1)}},$$
where  $C = C(k_i,\lambda_i,\beta)$.

We then turn to the case that $j < J_0$. We first use scaling equality $W_\infty(\overline{\nu}|_{A_j},\nu_n|_{A_j}) = W_\infty(j^{\beta}\overline{\nu}|_{A_j},j^{\beta}\nu_n|_{A_j})$. For simply notations, let $$\nu_{n,j} =j^{\beta}\nu_n|_{A_j},~~\overline{\nu}_j = j^{\beta}\overline{\nu}|_{A_j} .$$
Then the density function of $\overline{\nu}_j$ is defined by
$$\overline{\rho}_j(x) = j^{\beta} (1+\delta_j)\rho(x).$$

For every $k \in \mathbb{N}$, we partition $A_j$ into $2^k$ sub-domains. Each of them have a $\overline{\nu}_j-$mass of $\frac{1}{2^k}\overline{\nu}_j(A_j)$. Let $\mathcal{F}_{k,j}$ be the set of these sub-domains. $\mathcal{F}_{0,j} = A_j$. And $\mathcal{F}_{k+1,j}$ is obtained by bisecting each box in $\mathcal{F}_{k,j}$, according to $\overline{\nu}_{j}$. Thus, for any $Q \in \mathcal{F}_{k,j},$
$$\overline{\nu}_j(Q) = \frac{1}{2^k}\overline{\nu}_j(A_j),~~~~\nu(Q) = \frac{1}{2^k} \nu(A_j).$$

We define a series of new measures $\{\mu_{k,j}\}$ by setting $d\mu_{k,j}(x) = \rho_{k,j}(x) dx $ with
   $$\rho_{k,j}(x) = \frac{\nu_{n,j}(Q)}{\overline{\nu}_j(Q)}\cdot \overline{\rho}_j(x) = \frac{\nu_n(Q)}{\nu(Q)}j^{\beta}\rho(x),~~~~\forall~~x \in Q \in \mathcal{F}_{k,j}.$$
We claim that for $\forall Q \in \mathcal{F}_{k,j}$, $\forall ~k \leq k_n = \log_2\left(\frac{n\nu(A_j)}{10\log n }\right)$, there exists a constant $C$ such that the following inequality holds true with probability at least $1-2n^{-1}$
\begin{equation}\label{claim_2}
  W_\infty(\mu_{k,j}|_Q,\mu_{k+1,j}|_Q)\leq C \cdot (j+1)^{\beta} \cdot \left(\frac{\nu(A_j) \log n}{2^k n}\right)^{\frac12}.
\end{equation}

Assume that the claim holds.  Note that $\textrm{diam}(Q) = \int_Q dx \leq \int_Q (j+1)^{\beta}\rho(x)dx = (j+1)^{\beta}\nu(Q)$. Then for $j = 1,\cdots,J_0-1$, we have
   \begin{equation}\nonumber
     \begin{aligned}
       W_\infty(\nu_{n}|_{A_j},\overline{\nu}|_{A_j})  = W_{\infty}(\nu_{n,j},\overline{\nu_j}) &\leq \sum^{k_n}_{k = 1} W_\infty(\mu_{k-1,j},\mu_{k,j}) + W_\infty(\mu_{k_n,j},\nu_{n,j})\\
       & \leq \sum^{k_n}_{k=1}\bigg(C \cdot (j+1)^{\beta}\left(\frac{\nu(A_j) \log n}{2^k n }\right)^{\frac12}\bigg) + \max_{Q\in \mathcal{F}_k} \textrm{diam}(Q)\\
       & \leq \bigg(\sum^{k_n}_{k=1}\bigg(C \cdot \left(\frac{\nu(A_j) \log n}{2^k n }\right)^{\frac12}\bigg) + \frac{1}{2^{k_n}}\cdot \nu(A_j)\bigg) \cdot (j+1)^{\beta}\\
       &\leq C \left(\left(\frac{\log n}{n}\right)^{\frac12} + C \frac{\log n}{n}\right) (J_0+1)^{\beta}\\
       &= C \cdot \left(\frac{\log n}{n}\right)^{\frac12}\cdot \left\lfloor(\frac{n}{\log n})^{\frac{k_i}{2\beta(k_i+1)}}\right\rfloor^{\beta}\\
       & \leq C \left (\frac{\log n}{n}\right)^{\frac{1}{2(k_i +1)}}.
     \end{aligned}
   \end{equation}
Therefore, for $\forall j \in \mathbb{N}$,
$$W_\infty(\nu_{n}|_{A_j},\overline{\nu}|_{A_j})\leq C \left (\frac{\log n}{n}\right)^{\frac{1}{2(k_i +1)}}$$
where $C$ depends on $\beta, k_i, \lambda_i.$

Now we return to the proof of the claim (\ref{claim_2}). Actually, from the definition of $\mu_{k,j}$ and $\rho_{k,j}$ it follows that for $\forall~~x \in Q \in \mathcal{F}_{k,j},$
   $$\frac{\nu_n(Q)}{\nu(Q)}\cdot \frac{j^{\beta}}{(j+1)^{\beta}} \leq\rho_{k,j}(x)\leq \frac{\nu_n(Q)}{\nu(Q)},~~ $$
    and
   $$\mu_{k,j}(Q) = \mu_{k+1,j}(Q) = \nu_{n,j}(Q).$$

   Therefore, by Lemma \ref{infiniteW} we know that
   $$W_\infty(\mu_{k+1,j},\mu_{k,j})\leq \max_{Q\in \mathcal{F}_{k,j}}W_\infty(\mu_{k+1,j}|_Q,\mu_{k,j}|_Q).$$
   Let $Q_1$ be a sub-domain bisected from $Q$. Then $ Q_1 \in \mathcal{F}_{k+1,j}$. According to Lemma \ref{inf_den},
   \begin{equation*}
     \begin{aligned}
       W_\infty(\mu_{k+1,j}|_Q,\mu_{k,j}|_Q)  &= W_\infty\left(\frac{\nu(Q)}{\nu_n(Q)}\mu_{k+1,j}|_Q,\frac{\nu(Q)}{\nu_n(Q)}\mu_{k,j}|_Q\right)\\
         &\leq \frac{C}{\rho_{min}} \cdot \textrm{diam}(Q) \cdot  \left|\frac{\nu_n(Q_1)\nu(Q)}{\nu(Q_1)\nu_n(Q)} -1 \right| \cdot \left\|j^{\beta}\rho(x)\right\|_{L^\infty(Q)}\\
         & = \frac{C}{\rho_{min}} \cdot \textrm{diam}(Q) \cdot \left |\frac{2\nu_n(Q_1)}{\nu_n(Q)} -1 \right| \cdot \left\|j^{\beta}\rho(x)\right\|_{L^\infty(Q)}\\
         & \leq \frac{C}{\rho_{min}} \cdot \textrm{diam}(Q) \cdot  \left|\frac{2\nu_n(Q_1)}{\nu_n(Q)} -1\right |,
     \end{aligned}
   \end{equation*}
   where
   $$\rho_{min} = \frac{{j^{\beta}}}{(j+1)^{\beta}} \min\left\{1, \frac{\nu_n(Q_1)\nu(Q)}{\nu(Q_1)\nu_n(Q)}\right\} =\frac{{j^{\beta}}}{(j+1)^{\beta}} \min\left\{1, \frac{2\nu_n(Q_1)}{\nu_n(Q)}\right\} . $$

   To bound $W_\infty(\mu_{k+1,j}|_Q,\mu_{k,j}|_Q)$ , it suffices to estimate $\frac{1}{\rho_{min}}$ and $\left|\frac{2\nu_n(Q_1)}{\nu_n(Q)} -1 \right|$ respectively.  We first give a probabilistic estimate on $\frac{1}{\rho_{min}}$.

   Note that for $\forall Q \in \mathcal{F}_{k,j} $, $\frac{ n\nu_{n,j}(Q)}{j^{\beta}} \sim Bin(n,\nu(Q))$. Thus, we can use Bernstein's inequality and deduce that for $\forall k \leq k_n = \log_2\left(\frac{n\nu(A_j)}{10 \log n }\right)$,
     \begin{equation*}
       \begin{aligned}
        \mathbb{P}\Big(|\frac{\nu_{n,j}(Q)}{j^{\beta}} - \nu(Q)| \geq \frac12 \nu(Q)\Big) & \leq 2 \cdot \exp\Bigg(-\frac{\frac12 n\left(\frac{\nu(Q)}{2}\right)^2}{\nu(Q)\left(1-\nu(Q)\right)+\frac13\cdot\frac12\nu(Q)}\Bigg)\\
         &\leq 2 \cdot \exp\left(- \frac{1}{10}\cdot n \cdot \nu(Q)\right)\\
         & = 2 \cdot \exp\left(- \frac{1}{10}\cdot n \cdot \frac{1}{2^k}\nu(A_j)\right)\\
         & \leq  2n^{-1}.
       \end{aligned}
     \end{equation*}
     That is, with probability at least $1 - 2n^{-1}$,
     \begin{equation}\nonumber
       \left|\frac{1}{j^{\beta}}\cdot\frac{\nu_{n,j}(Q)}{\nu(Q)} - 1 \right| \leq \frac12.
     \end{equation}

     From the definition of $\nu_{k,j}$ we know
     \begin{equation}\label{geq12_2}
       \frac32 \geq \frac{\nu_{n}(Q)}{\nu(Q)} \geq \frac12,~~\frac32 \geq \frac{\nu_{n}(Q_1)}{\nu(Q_1)} \geq \frac12.
     \end{equation}
     Therefore,
     $$\frac{2\nu_n(Q_1)}{\nu_n(Q)} \geq \frac{\nu (Q_1)}{\nu_n(Q)} = \frac12 \cdot \frac{\nu (Q)}{\nu_n(Q)} \geq \frac13,$$
     and $\frac{1}{\rho_{min}}$ can be bounded with probability at least $1 - 2n^{-1}$
     $$ \frac{1}{\rho_{\min}} = \frac{1}{\frac{j^{\beta}}{(j+1)^{\beta}} \min\left\{1, \frac{2\nu_n(Q_1)}{\nu_n(Q)}\right\}} \leq \frac{3 (j+1)^{\beta}}{j^{\beta}} \leq 3\cdot 3 (1+\beta). $$

   We then estimate  $\left|\frac{2\nu_n(Q_1)}{\nu_n(Q)} -1 \right|$.

     Notice that if we set $m = n\cdot \frac{1}{j^{\beta}}\cdot \nu_{n,j}(Q)$,  then
     $$m\cdot \frac{\nu_{n,j}(Q_1)}{\nu_{n,j}(Q)} = \sum_1^n \delta_{X_i}(Q_1) \sim Bin\left(m,\frac{\nu(Q_1)}{\nu(Q)}\right) = Bin\left(m,\frac12\right).$$ Using Chernoff's inequality we get that
     \begin{equation}\nonumber
       \begin{aligned}
         \mathbb{P}\left(\left|\frac{\nu_{n,j}(Q_1)}{\nu_{n,j}(Q)} - \frac12\right| \geq \left(\frac{ 2^k \log n}{n \nu(A_j)}\right)^{\frac12}\right)
         & \leq 2\exp\left(-2\cdot \frac{m2^{k}\log n}{n\nu(A_j)} \right)\\
         & \leq 2\exp(-\log n)\\
         & =  2 n^{-1},
       \end{aligned}
     \end{equation}
     where the last inequality is obtained from (\ref{geq12_2}). Therefore, with probability at least $1-2n^{-1}$,
     $$\left|2\frac{\nu_{n}(Q_1)}{\nu_{n}(Q)} - 1\right| = \left|2\frac{\nu_{n,j}(Q_1)}{\nu_{n,j}(Q)} - 1\right|\leq 2 \left(\frac{2^k \log n}{n \nu(A_j)}\right)^{\frac12}.$$

   Finally, using the fact that $\textrm{diam}(Q) = \int_Q dx \leq  (j+1)^{\beta}\nu(Q),$ we know that with probability at least $1-2n^{-1}$,
   \begin{equation}\nonumber
     \begin{aligned}
       W_\infty(\mu_{k,j}|_Q,\mu_{k+1,j}|_Q) &\leq C \cdot \nu(Q) \left(\frac{ 2^k \log n}{n \nu(A_j)}\right)^{\frac12} \cdot(j+1)^{\beta} \\
       & \leq C \cdot (j+1)^{\beta} \cdot \left(\frac{\nu(A_j) \log n}{2^k n }\right)^{\frac12},
     \end{aligned}
   \end{equation}
which completes the proof of claim (\ref{claim_2}).

\bigskip

\textbf{\emph{Step 3:}} \
We then estimate $W_\infty(\overline{\nu}|_{B_i}, \widetilde{\nu}|_{B_i})$.
We first recall that $B_i = (x_i- \Delta_i,x_i+ \Delta_i)$ and
$$ d\widetilde{\nu}|_{B_i} = (1+\epsilon_i) \rho(x)dx~~,~~~~~~~~~d \overline{\nu}|_{B_i} = \sum_j \mathbbm{1}_{A_j}(1+\delta_j)\rho(x)dx,$$
where $\epsilon_i = \frac{\nu_n(B_i)}{\nu(B_i)}-1$  and $\delta_j = \frac{\nu_n(A_j)}{\nu(A_j)}-1$.
Let $T$ be the transportation map between $\overline{\nu}|_{B_i}$ and $ \widetilde{\nu}|_{B_i}$. Thus for any $x \in B_i$ and $y = Tx$,
$$\int^y_{x_i- \Delta_i} \widetilde{\rho}(s) ds = \int^x_{x_i- \Delta_i} \overline{{\rho}}(s) ds.$$
Without loss of generality, we assume $y>x$. Then
   \begin{equation}\label{est_rho_ba}
 \int^y_x \widetilde{\rho}(s) ds = \int^y_{x_i-\Delta_i} (\overline{\rho}(s) -\widetilde{\rho}(s)) ds \leq \int^y_{x_i-\Delta_i}|\widetilde{\rho}(s) - \overline{\rho}(s)| ds \leq \sum_j |\epsilon_i - \delta_j|\nu(A_j).
   \end{equation}

Let $S_n := n\nu_n(A_j)$. Then $S_n = \sum_{i=1}^{n} \delta_{X_i}(A_j) \sim Bin(n,\nu(A_j))$. According to Chebychev's inequality we know that
   \begin{equation}\nonumber
    \mathbb{P}\left(\frac{|S_n - n\cdot \nu(A_j)|}{\sqrt{n\nu(A_j)(1-\nu(A_j))}}  \geq \sqrt{\log n} \right) \leq (\log n)^{-1},
  \end{equation}
   which means that with probability at least $1- (\log n)^{-1}$,
   \begin{equation}\nonumber
     \frac{|n\nu_n(A_j) - n\cdot \nu(A_j)|}{\sqrt{n\nu(A_j)(1-\nu(A_j))}}  \leq \sqrt{\log n}.
   \end{equation}
  Then by the definition of $\delta_j$ we know that with probability at least $1- (\log n)^{-1}$,
   $$|\delta_j\nu(A_j)| \leq\Big(\frac{\log n \cdot \nu(A_j)(1-\nu(A_j))}{n}\Big)^{\frac12}.$$
  With a similar method we derive that with probability at least $1- (\log n)^{-1}$,
  \begin{align*}
    |\ep_i\nu(B_i)| \leq\Big(\frac{\log n \cdot \nu(B_i)(1-\nu(B_i))}{n}\Big)^{\frac12}.
  \end{align*}

   Note that in $A_j$, $\rho(s) \leq \frac{1}{j^{\beta}}$, which implies $\nu(A_j) = \int_{A_j}  \rho(s) ds \leq \int_{A_j}  \frac{1}{j^\beta} ds \leq \frac{1}{j^\beta}$. Therefore,
$$\sum_j |\delta_j|\nu(A_j) \leq \sum_j \left(\nu(A_j)\right)^{\frac12}\left(\frac{\log n}{n}\right)^{\frac12} \leq \sum_j \frac{1}{j^{\frac{\beta}{2}}}\left(\frac{\log n}{n}\right)^{\frac12} \leq C\cdot \left(\frac{\log n}{n}\right)^{\frac12}.$$
From the fact that $B_i = \bigsqcup_j A_j$ we know
   $$\sum_j |\epsilon_i| \nu(A_j) =  |\epsilon_i| \nu(B_i) \leq \Big(\frac{\log n \cdot \nu(B_i)(1-\nu(B_i))}{n}\Big)^{\frac12} \leq C\cdot \left(\frac{\log n}{n}\right)^{\frac12}.$$
Therefore from (\ref{est_rho_ba}) we derive that with probability at least $1- (\log n)^{-1}$,
  $$\int^y_x \widetilde{\rho}(s) ds \leq \sum_j |\epsilon_i - \delta_j|\nu(A_j) \leq \sum_j \big(|\epsilon_i | + |\delta_j|\big)\nu(A_j) \leq  C\cdot \left(\frac{\log n}{n}\right)^{\frac12}.$$

   Since in $B_i$, $\widetilde{\rho}(s) \geq (1+\epsilon_i)\underline{C}_i|x_i - s|^{k_i}$,  it follows that $\int^y_x \widetilde{\rho}(s) ds$ can be bounded from below in the following two cases respectively
\begin{equation}\nonumber
  \int^y_x \widetilde{\rho}(s) ds \geq
  \left\{
  \begin{aligned}
    &(1+\epsilon_i) \underline{C}_i \left[(x_i-x)^{k_i+1} +(y - x_i)^{k_i+1}\right] ,\quad &x_i \in (x,y),\\
    &(1+\epsilon_i) \underline{C}_i \left [(y-x)^{k_i+1}\right],\quad & x_i \notin (x,y).\\
  \end{aligned}
  \right.
\end{equation}
   The results are obtained by direct calculations so the proof is omitted here.
   In both cases, we can derive by lemma \ref{1lamma} that
   \begin{equation*}
   \begin{aligned}
      |y-x|
      & \leq C \bigg\{\bigg(\frac{ \left(\frac{\log n}{n}\right)^{\frac12}}{(1+\epsilon_i)\underline{C}_i }\bigg)^{\frac{1}{k_{i}+1}} \bigg\}.
   \end{aligned}
   \end{equation*}
   Therefore,
   $$W_\infty(\overline{\nu}|_{B_i}, \widetilde{\nu}|_{B_i}) \leq ||T-I||_{L^{\infty}(B_i)} \leq \max_{x\in B_i}|y-x| \leq C \left(\frac{\log n}{n}\right)^{\frac{1}{2(k_i+1)}},$$
where $C$ depends on $\ep_i , \underline{C}_i$ and $k_i$.

\bigskip

\textbf{\emph{Step 4:}} Finally, for $W_\infty(\nu,\widetilde{\nu})$, we use the same method as step 3 and deduce that
$$W_\infty(\nu,\widetilde{\nu}) \leq C \cdot\left( \frac{\log n}{n}\right)^{\frac12},$$
where $C $ depends on $\ep_i, k_i, \underline{C}_i$.

To sum up, with step 1-4, we know that
\begin{equation}
  \begin{aligned}
  W_\infty(\nu,\nu_n)&\leq W_\infty(\nu,\widetilde{\nu}) + \max \left\{ \max_{i=1,\cdots,N}\big[W_\infty(\widetilde{\nu}|_{B_i},\overline{\nu}|_{B_i}) +\sup_j W_\infty(\overline{\nu}|_{A_j},\nu_n|_{A_j})\big], W_\infty(\widetilde{\nu}|_{B_{N+1}},\nu_n|_{B_{N+1}})\right\}\\  &\leq C\cdot \max_i \left(\frac{\log n}{n}\right)^{\frac{1}{2(k_i+1)}}.\notag
\end{aligned}
\end{equation}
where $C$ depends on $ k_i $ and $\underline{C}_i$.
This completes the proof of Theorem \ref{w_2}.
\end{proof}

\bigskip

\subsection{Proof of Corollary \ref{w_3}}

\begin{proof}
Let $A = \{x: \rho(x) < 1 \}$, $B = \{x: \rho(x) \geq 1\}$ and assume that they both are connected sets( otherwise we can divide them into connected sets).

   Define a probability measure on $D$ by $ d \widetilde{\nu} = \left((1+\epsilon_A) \mathbbm{1}_A \rho(x) + (1+\epsilon_B) \mathbbm{1}_B \rho(x)\right) dx$, where
   $$\epsilon_A = \frac{\nu_n(A)}{\nu(A)} - 1, ~~  ~~\epsilon_B = \frac{\nu_n(B)}{\nu(B)} - 1.$$
Thus, it's easy to see that
\begin{equation}\label{cor_nu}
  \widetilde{\nu}(A) = \nu_n(A) ~~~~\textrm{and} ~~~~~\widetilde{\nu}(B) = \nu_n(B).
\end{equation}

In order to estimate $W_{\infty}(\nu,\nu_n)$, it suffices to estimate $W_{\infty}(\nu, \widetilde{\nu})$ and $W_{\infty}(\widetilde{\nu}, \nu_n)$ respectively.

\bigskip

\textbf{\emph{Step 1:}} \
We first estimate $W_\infty(\nu_n,\widetilde{\nu})$. Using Lemma \ref{infiniteW} and (\ref{cor_nu}) we know that
   \begin{equation}\nonumber
     \begin{aligned}
       W_{\infty}(\widetilde{\nu}, \nu_n) & \leq  \max\left\{W_{\infty}\left(\widetilde{\nu}|_A, \nu_n|_A\right),W_{\infty}\left(\widetilde{\nu}|_B, \nu_n|_B\right)\right\} \\
       & = \max\left\{W_{\infty}\left(\frac{1}{\widetilde{\nu}(A)}\widetilde{\nu}|_A, \frac{1}{\widetilde{\nu}(A)}\nu_n|_A\right),W_{\infty}\left(\frac{1}{\widetilde{\nu}(B)}\widetilde{\nu}|_B, \frac{1}{\widetilde{\nu}(B)}\nu_n|_B\right)\right\}.
     \end{aligned}
   \end{equation}

Note that $$\frac{1}{\widetilde{\nu}(A)}\widetilde{\nu}|_A = \frac{1}{\nu(A)}\nu|_A$$ and
   $$\frac{1}{\widetilde{\nu}(A)}\nu_n|_A = \frac{1}{n\widetilde{\nu}(A)} \sum^n_{i = 1} \delta_{X_i}|_A =  \frac{1}{n \nu_n(A)} \sum^n_{i = 1} \delta_{X_i}|_A = \frac{1}{\sum^n_{i = 1} \delta_{X_i}(A)}\sum_{X_i\in A} \delta_{X_i}|_A.$$
   Therefore, $\frac{1}{\widetilde{\nu}(A)}\nu_n|_A$ is the empirical measure of $\frac{1}{\widetilde{\nu}(A)}\widetilde{\nu}|_A$. By Theorem \ref{w_1} we know that
   $W_{\infty}(\widetilde{\nu}|_A, \nu_n|_A)\leq C \cdot \left(\frac{\log n}{n}\right)^{\frac12} $.

   Similarly, we can deduce that $W_{\infty}(\widetilde{\nu}|_B, \nu_n|_B)\leq C\cdot\max_i \left(\frac{\log n }{n}\right)^{\frac{1}{2(k_i+1)}}$. Therefore,
   $$W_{\infty}(\widetilde{\nu},\nu_n)\leq C\cdot\max_i \left(\frac{\log n }{n}\right)^{\frac{1}{2(k_i+1)}}. $$

\bigskip

\textbf{\emph{Step 2:}} \
We then estimate $W_\infty(\widetilde{\nu}, \nu)$.

   Let T be the transportation map between  $\widetilde{\nu}$ and $\nu.$ Then for $\forall x \in D$ and $y = Tx $,
   $$\int^x_{0} \widetilde{\rho}(s) ds = \int^y_{0} {\rho}(s) ds.$$
   Without loss of generality, we assume $y>x$. Then it follows that
   $$
     \int^y_x \rho(s) ds = \int^x_{0} \widetilde{\rho}(s) - \rho(s) ds  \leq \int^x_{0}|\widetilde{\rho}(s) - \rho(s)| ds  \leq |\epsilon_A|\nu(A) + |\epsilon_B|\nu(B).
   $$
   By Chebychev's inequality we know that with probability at least $1 - (\log n)^{-1}$,
   \begin{align*}
     |\epsilon_A|\nu(A) + |\epsilon_B|\nu(B) \leq C \left(\frac{\log n}{n}\right)^{\frac12}.
   \end{align*}
    Thus,
   $$\int_{(x,y)\cap A} \rho(s)ds + \int_{(x,y)\cap B} \rho(s)ds \leq C \cdot \left(\frac{\log n}{n}\right)^{\frac12}.$$

   By the same method in the proof of Theorem  \ref{w_2}, we can give a lower bound on $\int_{(x,y)\cap A} \rho(s)ds$ and $\int_{(x,y)\cap B} \rho(s)ds$ respectively and conclude that with probability at least $1 - (\log n)^{-1}$,
   $$W_\infty(\nu, \nu_n) \leq C \cdot \max_i \left(\frac{\log n }{n}\right)^{\frac{1}{2(k_i+1)}}.$$
This completes the proof of Corollary \ref{w_3}.
\end{proof}

\bigskip

\begin{rem}
We showed the rate of convergence of $\nu_n$ to $\nu$ when the density function $\rho(x)$ is unbounded in one dimension. We expect that similar results also hold to be true in high dimensions.
However, the idea of the proof needs to be adapted. In particular, the estimate of  $W_{\infty}(\widetilde{\nu},\nu )$
becomes quite technical in high dimensions, where $\tilde{\nu}$ is an auxiliary measure introduced in \eqref{nu:widetilde} for the purpose of removing the mass discrepancy between $\nu$ and $\nu_n$ in local regions. In fact, in one dimension we estimate $W_{\infty}(\nu,\widetilde{\nu})$ by using that
\begin{equation*}
  W_\infty(\nu,\widetilde{\nu}) \leq ||T-I||_{L^{\infty}} \leq \max_{x\in D}|y-x| \leq \frac{1}{\lambda} \int^y_x \rho(s)ds,
\end{equation*}
where $T$ is the transportation map between $\widetilde{\nu}$ and $\nu$ and $y = Tx$.  In high dimensions, it is not clear to us how to bound $W_{\infty}(\nu,\widetilde{\nu})$ in terms of certain integral of the density. This is to be investigated in our future work.
\end{rem}

\section*{Acknowledgements}
The research was partially supported by KI-Net NSF RNMS11-07444 and NSF DMS-1812573.
The authors would like to thank Johannes Wiesel for suggesting the use of the Dvoretzky-Kiefer-Wolfowitz inequality to prove Theorem \ref{w_1}, which simplifies an earlier version of the proof of the theorem.

\printindex

\newpage

\bibliographystyle{amsplain}
\bibliography{bibfile}

\providecommand{\bysame}{\leavevmode\hbox to3em{\hrulefill}\thinspace}
\providecommand{\MR}{\relax\ifhmode\unskip\space\fi MR }
\providecommand{\MRhref}[2]{%
  \href{http://www.ams.org/mathscinet-getitem?mr=#1}{#2}
}
\providecommand{\href}[2]{#2}
\begin{thebibliography}{10}

\bibitem{Ajtai1984}
M.~Ajtai, J.~Koml{\'o}s, and G.~Tusn{\'a}dy, \emph{On optimal matchings},
  Combinatorica \textbf{4} (1984), no.~4, 259--264.

\bibitem{barthe2013combinatorial}
Franck Barthe and Charles Bordenave, \emph{Combinatorial optimization over two
  random point sets}, S{\'e}minaire de Probabilit{\'e}s XLV, Springer, 2013,
  pp.~483--535.

\bibitem{Bernstein1927}
S.N. Bernstein, \emph{The theory of probabilities}, Gastehizdat Publishing
  House,Moscow, 1946.

\bibitem{boissard2011simple}
Emmanuel Boissard et~al., \emph{Simple bounds for the convergence of empirical
  and occupation measures in 1-wasserstein distance}, Electronic Journal of
  Probability \textbf{16} (2011), 2296--2333.

\bibitem{boissard2014mean}
Emmanuel Boissard, Thibaut Le~Gouic, et~al., \emph{On the mean speed of
  convergence of empirical and occupation measures in wasserstein distance},
  Annales de l'Institut Henri Poincar{\'e}, Probabilit{\'e}s et Statistiques,
  vol.~50, Institut Henri Poincar{\'e}, 2014, pp.~539--563.

\bibitem{bolley2007quantitative}
Fran{\c{c}}ois Bolley, Arnaud Guillin, and C{\'e}dric Villani,
  \emph{Quantitative concentration inequalities for empirical measures on
  non-compact spaces}, Probability Theory and Related Fields \textbf{137}
  (2007), no.~3-4, 541--593.

\bibitem{champion2008wasserstein}
Thierry Champion, Luigi De~Pascale, and Petri Juutinen, \emph{The
  $\infty$-wasserstein distance: Local solutions and existence of optimal
  transport maps}, SIAM Journal on Mathematical Analysis \textbf{40} (2008),
  no.~1, 1--20.

\bibitem{Chernoff_1952}
Herman Chernoff, \emph{A measure of asymptotic efficiency for tests of a
  hypothesis based on the sum of observations}, The Annals of Mathematical
  Statistics \textbf{23} (1952), no.~4, 493--507.

\bibitem{d2006using}
Michele d'Amico, Patrizio Frosini, and Claudia Landi, \emph{Using matching
  distance in size theory: A survey}, International Journal of Imaging Systems
  and Technology \textbf{16} (2006), no.~5, 154--161.

\bibitem{dudley1969speed}
RM~Dudley, \emph{The speed of mean glivenko-cantelli convergence}, The Annals
  of Mathematical Statistics \textbf{40} (1969), no.~1, 40--50.

\bibitem{fournier2015rate}
Nicolas Fournier and Arnaud Guillin, \emph{On the rate of convergence in
  wasserstein distance of the empirical measure}, Probability Theory and
  Related Fields \textbf{162} (2015), no.~3-4, 707--738.

\bibitem{gibbs2002choosing}
Alison~L Gibbs and Francis~Edward Su, \emph{On choosing and bounding
  probability metrics}, International statistical review \textbf{70} (2002),
  no.~3, 419--435.

\bibitem{Leighton1989}
T.~Leighton and P.~Shor, \emph{Tight bounds for minimax grid matching with
  applications to the average case analysis of algorithms}, Combinatorica
  \textbf{9} (1989), no.~2, 161--187.

\bibitem{shor1991}
P.~W. Shor and J.~E. Yukich, \emph{Minimax grid matching and empirical
  measures}, Ann. Probab. \textbf{19} (1991), no.~3, 1338--1348.

\bibitem{chebichef1867}
P.~Tchebichef, \emph{Des valeurs moyennes}, Journal de math¨¦matiques pures et
  appliqu¨¦es \textbf{12} (1867), no.~2, 177--184.

\bibitem{trillos2016variational}
Nicol{\'a}s~Garc{\'\i}a Trillos and Dejan Slep{\v{c}}ev, \emph{A variational
  approach to the consistency of spectral clustering}, Applied and
  Computational Harmonic Analysis (2016).

\bibitem{trillos2014rate}
Nicol{\'a}s~Garcia Trillos and Dejan Slep\v{c}ev, \emph{On the rate of
  convergence of empirical measures in $\infty$-transportation distance},
  Canadian Journal of Mathematics \textbf{67} (2014), 1358.

\bibitem{van1996weak}
Aad~W Van Der~Vaart and Jon~A Wellner, \emph{Weak convergence}, Weak
  convergence and empirical processes, Springer, 1996, pp.~16--28.

\bibitem{von2007tutorial}
Ulrike Von~Luxburg, \emph{A tutorial on spectral clustering}, Statistics and
  computing \textbf{17} (2007), no.~4, 395--416.

\end{thebibliography}

\end{document}